\documentclass{article}

\usepackage{amsmath,amsfonts,amsthm,amssymb,amscd,color,xcolor,mathrsfs,verbatim,microtype}
\usepackage{graphicx,eurosym}
\usepackage{amssymb,amsmath,comment}
\usepackage{amsfonts,mathrsfs}
\usepackage{color}
\usepackage{bbm}
\usepackage{bm}
 \usepackage{hyperref}
\usepackage{stmaryrd} 

\usepackage[applemac]{inputenc}

\usepackage[cyr]{aeguill}

\colorlet{darkblue}{blue!50!black}

\hypersetup{
    colorlinks,%
    citecolor=blue,%
    filecolor=red,%
    linkcolor=darkblue,%
    urlcolor=blue,%
    pdfnewwindow=true,%
    pdfstartview={FitH}
}

\usepackage{graphicx,amscd,mathrsfs,wrapfig,mathrsfs,lipsum}
\usepackage{eufrak}
\usepackage{float}
\usepackage{tikz}
\usepackage{multicol}
\usepackage{caption}
\usetikzlibrary{arrows}
\usepackage{capt-of}

\colorlet{darkblue}{blue!50!black}

\usepackage{marginnote}

\newcommand{\p}{\partial}
\newcommand{\e}{\varepsilon}

\newcommand{\mek}{{\mathbf 1}}

\newcommand{\R}{{\mathbb R}}
\newcommand{\Z}{{\mathbb Z}}
\newcommand{\IP}{{\mathbb P}}
\newcommand{\pP}{{\mathbb P}}
\newcommand{\I}{{\mathbb I}}
\newcommand{\E}{{\mathbb E}}
\newcommand{\T}{{\mathbb T}}

\newcommand{\la}{\lambda}

\newcommand{\ty}{\infty}

\newcommand{\de}{\delta}

\newcommand{\wwww}{{\mathfrak w}}

\newcommand{\mm}{{\mathfrak m}}

\newcommand{\BBBB}{{\mathfrak B}}

\newcommand{\aA}{{\cal A}}

\newcommand{\FF}{{\cal F}}

\newcommand{\HH}{{\cal H}}

\newcommand{\KK}{{\cal K}}

\newcommand{\MM}{{\cal M}}

\newcommand{\PP}{{\cal P}}

\def\cF{{\mathcal F}}

\def\cD{{\mathcal D}}
\def\cM{{\mathcal M}}
\def\cA{{\mathcal A}}

\def\eps{\varepsilon}
\def\cL{{\mathcal L}}

\newcommand{\lag}{\langle}
\newcommand{\rag}{\rangle}

\newcommand{\dd}{{\textup d}}

\newcommand{\PPPP}{{\mathfrak P}}

\newcommand{\BBBBB}{{\mathcal B}}

\newcommand{\lspan}{\mathop{\rm span}\nolimits}

\newcommand{\diver}{\mathop{\rm div}\nolimits}

\theoremstyle{plain}

\newtheorem*{mtheorem}{Main Theorem}

\newtheorem*{lemma*}{Lemma}
\newtheorem{theorem}{Theorem}[section]
\newtheorem{lemma}[theorem]{Lemma}
\newtheorem{proposition}[theorem]{Proposition}

\theoremstyle{definition}

\theoremstyle{remark}

\numberwithin{equation}{section}

\begin{document}

\author{Vahagn~Nersesyan\,\footnote{NYU-ECNU Institute of Mathematical Sciences, NYU Shanghai, 3663 Zhongshan Road North, Shanghai, 200062, China, e-mail: \href{mailto:vahagn.nersesyan@nyu.edu}{Vahagn.Nersesyan@nyu.edu}}\, \footnote{Universit\'e Paris-Saclay, UVSQ, CNRS, Laboratoire de Math\'ematiques de Versailles, 78000, Versailles, France}\and
Xuhui Peng\,
\footnote{MOE-LCSM, School of Mathematics and Statistics, Hunan Normal University, Changsha, Hunan 410081, China;  e-mail: \href{mailto:xhpeng@hunnu.edu.cn}{xhpeng@hunnu.edu.cn}}
\and
Lihu  Xu\,\footnote{Department of Mathematics, Faculty of Science and Technology, University of Macau, Macau, China;  e-mail: \href{mailto:lihuxu@umac.mo}{lihuxu@umac.mo}} }

\title{Large deviations principle via Malliavin calculus for the    Navier--Stokes system driven by a degenerate  white-in-time noise}
\date{\today}
\maketitle

\begin{abstract}

  The purpose of this paper is to   establish the Donsker--Varadhan type large deviations principle (LDP) for the two-dimensional stochastic Navier--Stokes system.~The main novelty   is that the noise is   assumed to be highly     degenerate in the Fourier space.~The~proof is carried out by using a criterion   for the LDP developed  in~\cite{JNPS-2018}
   in a discrete-time setting and extended   in~\cite{MN-2015} to the continuous-time.~One of the main  conditions  of that criterion is the   uniform Feller~property for the   Feynman--Kac semigroup, which we verify    by  using    Malliavin calculus.

\smallskip
\noindent
{\bf AMS subject classification:}  35Q30, 37H15, 60B12, 60F10,  60H07, 93B05

\smallskip
\noindent
{\bf Keywords:}      Navier--Stokes system, large deviations,     Malliavin calculus,   degenerate noise, Feynman--Kac semigroup, approximate controllability,  multiplicative ergodicity, uniform Feller property

\end{abstract}

 \tableofcontents

\setcounter{section}{-1}

\section{Introduction}\label{S:0}

 In this paper, we   study   the large deviations principle (LDP)~for~the incompressible Navier--Stokes (NS) system on the torus~$\T^2=\R^2/2\pi\Z^2$:
\begin{equation} \label{0.1}
\p_t u-\nu\Delta u+\langle u,\nabla \rangle u+\nabla p=\eta(t,x), \quad\diver u=0, \quad x\in \T^2.
\end{equation}Here $u = (u_1(t,x), u_2(t,x))$  and~$p = p(t, x)$ are the unknown velocity field and   pressure of the fluid,
$\nu>0$ is the viscosity, and $\eta$ is   an external random force.     We consider this system      in the  usual    space
\begin{equation}\label{0.2}
H=\Bigg\{u\in L^2(\T^2,\R^2): \, \int_{\T^2} u(x) \dd x=0, \,\,\, \diver u=0 \,\, \text{in }\T^2   \Bigg\}
\end{equation}
   endowed    with the   $L^2$-scalar product $\lag \cdot, \cdot\rag$ and the corresponding norm $\|\cdot\|.$
Projecting the system \eqref{0.1}  to the space~$H$, we   eliminate the pressure term and obtain the evolution~equation
\begin{equation} \label{0.3}
\p_t u-\nu \Delta  u +\Pi(\langle u,\nabla \rangle u) = \Pi \eta,
\end{equation}where $\Pi$ is the Leray    projection  to  $H$ in~$L^2(\T^2,\R^2)$ (see Section~6 in Chapter~1 of \cite{lions1969}).
 We   assume that~$\eta$ is a white-in-time~noise of the form
\begin{equation}\label{0.4}
\eta(t,x)=  \p_t W(t,x),\quad W(t,x)=\sum_{l\in \KK} b_lW_l(t) e_l(x),
\end{equation}where $\KK\subset \Z^2_*$ is a  finite set,
 $\{b_l\}_{l\in \KK}$ are non-zero   real numbers,     $\{W_l\}_{l\in \KK}   $ are   independent standard
       Brownian motions     on a filtered probability space~$(\Omega,\FF,  \{\FF_t\}, \pP)$    satisfying  the usual conditions (see   Definition 2.25 in~\cite{KS1991}),  and
       $$
e_l(x)= \begin{cases}  l^\bot\cos\lag l,x\rag  & \text{if }l_1>0\text{ or } l_1=0,\, l_2>0, \\  l^\bot \sin\lag l,x\rag   & \text{if }l_1<0\text{ or } l_1=0,\, l_2<0 , \quad l =(l_1,l_2)\end{cases}
$$
  with  $l^\bot=(- l_2, l_1).$
  In other words,   $\KK$ is  the collection of the  Fourier modes    directly perturbed by the noise.~Under the above   assumptions, the NS system~\eqref{0.3}
     defines a family of   Markov processes~$(u_t,\IP_u)$ parametrised by the initial condition $u(0)=u \in H$.~The~ergodic properties of this family  have been extensively studied in the literature.   It~is~now well known that~$(u_t,\IP_u)$  admits a unique and exponentially mixing  stationary measure, provided that the set $\KK$  is sufficiently large.~Under the condition that~$\KK$ contains all the determining modes, the ergodicity has~been~established in different settings in
      the papers~\cite{FM-1995,KS-cmp2000,EMS-2001, KS-jmpa2002, BKL-2002}.~Later, it~was shown that   the ergodicity   remains true for much smaller set $\KK$; see~the papers~\cite{HM-2006, HM-2011,FGRT-2015}
      for the case when the noise  is white-in-time   and~\cite{KNS-2018, KNS-2019} for the case of  a
        general bounded noise.~The~reader is referred to the book~\cite{KS-book} for more references and for  detailed description of different  methods.

 In this paper, we study the Donsker--Varadhan type LDP for the NS system~\eqref{0.3}.~This type of LDP has been extensively   studied in the case of finite-dimensional diffusions and   Markov processes in compact spaces;
    see the papers~\cite{DV-1975},  the  books~\cite{FW1984,DS1989,ADOZ00}, and the  references therein.~The paper~\cite{Wu-2001} established a general criterion for Donsker--Varadhan type    LDP  for   Markov processes that are strong Feller and irreducible. In that paper the criterion is applied   to a class of stochastic damping Hamiltonian systems. There are only few papers considering the problem of      LDP for   randomly forced  PDEs.~The first results are obtained in~\cite{gourcy-2007b, gourcy-2007a} in the case of  the stochastic Burgers  and   NS equations with strong assumptions on~the decay of the coefficients~$\{b_l\}$.~Indeed,  these papers use the criterion of~\cite{Wu-2001}, so they require   some lower bounds  for   $\{b_l\}$  in order    to  guarantee the strong~Feller~property. These   assumptions     have been   relaxed to the conditions~$b_l\neq 0$ for all~$l\in \Z^2_*$ and~$\sum_{l\in \Z^2_*}|l||b_l|^2<+\ty $
 in the papers~\cite{JNPS-cpam2015, JNPS-2018}, where a family of dissipative PDEs is considered driven by a random kick-force.   The proofs of these papers are based on a study~of the long-time behaviour of Feynman--Kac semigroup   and    a Kifer type    criterion for  the  LDP.~Under similar   non-degeneracy  conditions, the local LDP is proved in~\cite{MN-2015} for the  stochastic   damped nonlinear wave equation, and the full LDP    is proved in~\cite{VN-2019} for the stochastic NS system.~A~controllability approach is used in~\cite{JNPS-2019} to prove   the LDP   for the Lagrangian trajectories of the NS system. Recently,
the  criterion of~\cite{Wu-2001} has been used in~\cite{WX-2018} in the case of   SPDEs driven by stable type noises and in~\cite{WXX-2021} in the case of non-linear monotone SPDEs with white-in-time noise.

 \smallskip

All the papers  mentioned above establish the LDP under the assumption that  the noise is non-degenerate, i.e., perturbs directly all the Fourier modes in the equation.~The goal of the present paper is to establish the LDP in the case of a highly degenerate noise, i.e., when only few  Fourier modes are directly perturbed.~To formulate our main result, let us recall that a set $\KK\subset \Z^2_*$
is a generator if any element of~$\Z^2$
is a finite linear
combination of elements of~$\KK$ with integer coefficients. In~what follows, we    assume that the following condition~is~satisfied.
\begin{description}
\item[\hypertarget{H}]{\bf (H)}
{\sl The set  $\KK\subset \Z^2_*$ in \eqref{0.4} is a finite symmetric (i.e.,  $-\KK=\KK$)  generator
that contains at least two non-parallel vectors~$m$ and~$n$ such that $|m|\neq |n|$.}
\end{description}
This is   the   condition under which   the  ergodicity of the NS system is  established  in~\cite{HM-2006, HM-2011} in the case of a white-in-time noise and in \cite{KNS-2018} in the case of a bounded   noise. The set
$$
\KK=\{(1,0), \, (-1,0), \, (1,1), \, (-1,-1)\}\subset \Z^2_*
$$ is an   example satisfying this condition.

  For any $u\in H$, let us define the   family of  occupation~measures  \begin{equation} \label{0.5}
\zeta_t=\frac1t\int_{0}^{t}\delta_{u_s}   \dd    s,  \,\,\,  t>0
\end{equation}on the probability space  $(\Omega,\FF,\IP_u)$, where
$\de_v$ is the Dirac measure concentrated at $v\in H$.
 \begin{mtheorem}
     Under the Condition~\hyperlink{H}{\rm(H)},    the family~$\{\zeta_t,t>0\}$         satisfies the~LDP.
       \end{mtheorem}
See Theorem~\ref{T:1.1} for   more detailed formulation of this result.~The proof   is   carried out by using a criterion     for the LDP developed  in~\cite{JNPS-2018}
   in a discrete-time setting and extended   in~\cite{MN-2015}  to the continuous-time.~According to that criterion, the LDP will be established if we show that the following five properties hold for the  Feynman--Kac semigroup  associated with the NS system~\eqref{0.3}:
   growth~properties,   existence of  eigenvector, time-continuity, uniform irreducibility, and uniform Feller property.~The first three properties are verified in~\cite{VN-2019} and they hold no matter   how  degenerate is the noise. The~uniform irreducibility~property follows from the approximate controllability results obtained in~\cite{AS-2005, AS-2006}.~It is interesting to note that  Condition~\hyperlink{H}{\rm(H)}  is necessary and sufficient for the approximate controllability of the NS system if one uses controls acting via the
     Fourier modes in $\KK$.~The main technical difficulty of this paper is related to the verification of the uniform Feller property, which we carry out    by developing
  the Malliavin calculus analysis of the papers~\cite{MP-2006, HM-2006, HM-2011}.
  More precisely, we derive the uniform Feller property from a gradient estimate for the
  Feynman--Kac semigroup.~The proof of latter   contains essential differences with respect to the situations studied in~\cite{MP-2006, HM-2006, HM-2011}   because of the  non-Markovian  character of the Feynman--Kac semigroup.

  \smallskip

    This paper is organised as follows. In Section \ref{S:1}, we
    explain how the Main Theorem is derived from the above-mentioned five properties. In Section \ref{S:2}, we recall some elements of   Malliavin calculus, and in Section \ref{S:3}, we  verify the uniform Feller property.

  \subsubsection*{Acknowledgement}
  Xuhui Peng would like to gratefully thank University  of  Macau for the hospitality, his research is supported in part by
  by NNSFC (No. 12071123), Scientific Research
Project of Hunan Province Education Department (No. 20A329) and  Program of Constructing Key Discipline in Hunan Province. Lihu Xu is supported in part by NNSFC grant (No.12071499), Macao S.A.R grant FDCT 0090/2019/A2 and University of Macau grant  MYRG2018-00133-FST.

 \subsubsection*{Notation}

  In this paper, we use the following  notation.

\smallskip
\noindent
$H$ is the space of divergence-free square-integrable vector fields on $\T^2$ with zero mean value (see~\eqref{0.2}).
It is endowed with the $L^2$-norm $\|\cdot\|$.

\smallskip
\noindent
$H^m= H^m(\T^2, \R^2)\cap H$, where $H^m(\T^2, \R^2)$ is the       Sobolev space of order~$m\ge1$. We endow the space  $H^m$ with the usual Sobolev norm~$\|\cdot \|_m$.

\smallskip
\noindent
$B_{H^m}(a,r)$   is the closed   ball in $H^m$ of radius~$r>0$ centred at~$a$.~We write~$B_{H^m}(r)$ when~$a=0$.

\smallskip
\noindent
 We   consider the NS system in the vorticity formulation in the space of  square integrable      zero  mean   functions:
\begin{align}
\label{0.6}
\tilde{H}=\left\{w\in L^2(\T^2,\R):\int_{\T^2}w(x)\dd  x=0\right\}
\end{align}
 equipped  with the $L^2$-norm $ \|\cdot\|.$
Let $\tilde H^m= H^m(\T^2, \R)\cap \tilde H$, $m\ge1$ be endowed with the Sobolev norm   denoted by $\|\cdot \|_m$.

\smallskip
\noindent
$L^\ty(H)$  is the space of bounded Borel-measurable functions $\psi:H\to\R$   with the norm $\|\psi\|_\ty=\sup_{u\in H}|\psi(u)|$.~$C_b(H)$ is the space of continuous functions $\psi\in L^\ty(H)$.~$C^1_b(H)$ is the space of   functions~$\psi\in C_b(H)$ that are continuously Fr\'echet differentiable with bounded derivative.

\smallskip
\noindent
Let $\wwww:H\to[1,+\ty]$ be a Borel-measurable function. Then $C_\wwww(H)$ (resp.,~$L_\wwww^\ty(H)$) is the space of continuous (resp., Borel-measurable) functions $\psi:H\to\R$ such that
$\|\psi\|_{L_\wwww^\ty}=\sup_{u\in H} |\psi(u)|/\wwww(u)<+\ty.$

 \smallskip
\noindent
  $\MM_+(H)$ is the collection   of non-negative finite Borel measures on $H$ endowed~with the weak convergence topology.~For any~$\psi\in L^\ty(H)$ and  $\mu\in \MM_+(H)$,  we write~$\lag \psi,\mu\rag=\int_H \psi(u)\mu(\dd u).$
 $\PP(H)$ is the subset of probability   measures, and~$\PP_\wwww(H)$ is  the set of $\mu\in\PP(H)$ such that $\lag \wwww,\mu\rag<+\ty$.

\smallskip
\noindent
$\cL(X,Y)$ is the space of linear bounded operators between Banach spaces $X$ and~$Y$ endowed with the natural norm $\|\cdot\|_{\cL(X,Y)}$.

\smallskip
\noindent
  The   letter $C$ is used to denote    unessential constants that can change from line to line.

 \section{Main results}\label{S:1}

\subsection{LDP and multiplicative ergodicity}\label{S:1.1}

    Recall that  a mapping $I : \PP(H) \to [0, +\ty]$ is  a   good rate function  if   its level~sets
    $$
    \left\{\sigma\in\PP(H) : I(\sigma) \le \alpha \right\}, \quad \alpha \ge0
    $$ are compact. Moreover, if the effective domain of $I$, defined by
$$
D_I=\left\{\sigma\in\PP(H):I(\sigma)<+\ty \right\},
$$ is not a singleton,     we say that $I$ is a non-trivial good rate function.
For any~$\gamma>0$ and $M>0$, we   set
$$
\Lambda(\gamma,M) : = \Bigg \{\nu \in \PP(H) : \int_H e^{\gamma\|u\| ^2 }\nu ( \dd u) \le M \Bigg\}.
$$
 The following is a more   detailed version of the Main Theorem formulated in the Introduction.
\begin{theorem}  \label{T:1.1}
Assume that Condition~\hyperlink{H}{\rm(H)} is verified.~Then, for any~$\gamma>0$  and~$M>0$, the family of random probability measures~$\{\zeta_t,t>0\}$ defined by~\eqref{0.5}        satisfies the LDP uniformly w.r.t. the initial measure $\nu\in \Lambda(\gamma,M)$. More~precisely,  there is a non-trivial 		good 	rate 	function $I: \PP(H)\to [0,+\ty]$  that does not depend on   $\gamma$ and  $M$ and satisfies the inequalities   \begin{gather*}
	\limsup_{t\to+\ty}    \frac{1}{t} \log \sup_{\nu\in \Lambda(\gamma,M)} \pP_\nu\left\{\zeta_t\in F\right\}\le -\inf_{\sigma\in F} I(\sigma),\\
\liminf_{t\to+\ty} \frac{1}{t}  \log \inf_{\nu\in \Lambda(\gamma,M)}\pP_\nu\left\{\zeta_t\in G\right\}\ge -\inf_{\sigma\in G}I(\sigma)
\end{gather*}for any closed set $F$ and any open set $G$ in $\PP(H)$.
 \end{theorem}
 This theorem is derived from   a     multiplicative ergodic theorem for the NS system. To formulate that result, let us introduce    the following   weight functions:
\begin{align*}
\mm_\gamma(u)&=\exp\left(\gamma  \|u\|^2\right), \quad \gamma>0,\\
\wwww_m(u)&=1+\|u\|^{2m},\quad m\ge1,\,\, u\in H.
 \end{align*}There is a constant  $\gamma_0=\gamma_0(\BBBB_{0})>0$, where  $\BBBB_{0}=\sum_{l\in \KK}b_l^2$, such that
\begin{align}
\E_u \mm_\gamma(u_t) &\le e^{-\gamma t} \mm_\gamma(u)+C,\label{1.1}
\\
 \E_u \wwww_m(u_t) &\le e^{-2m  t} \wwww_m(u)+C \label{1.2}
\end{align}   for any $\gamma\in(0,\gamma_0)$, $m\ge1$,  $u\in H$, and $t\ge0$,
 where $C=C(m,\varkappa, \BBBB_{0})>0$ is a constant; e.g.,~see Proposition~2.4.9 in~\cite{KS-book} and Lemma~5.3 in~\cite{VN-2019} for a proof of these inequalities.

  For any $V\in C_b(H),$
 the     Feynman--Kac semigroup  associated with the Markov family $(u_t,\pP_u)$   is defined by
$$
\PPPP_t^V\psi(u)=\E_u \left\{ \exp\left(\int_{0}^tV(u_s) \dd s\right) \psi(u_t)\right\}, \quad \PPPP_t^V: C_b(H)\to C_b(H);
$$  its dual is denoted by $\PPPP_t^{V*}:\MM_+(H)\to \MM_+(H)$.~From~\eqref{1.1}    it~follows that~$\PPPP_t^V$ maps the space $C_{\mm_\gamma}(H)$ into itself
for   $\gamma\in(0,\gamma_0)$.
\begin{theorem}  \label{T:1.2}
Assume that Condition~\hyperlink{H}{\rm(H)} is verified and $   V\in C_b^1(H)$.~Then   there are constants      $ m=m(V)\ge1$ and  $\gamma=\gamma( \BBBB_{0})\in (0,\gamma_0)$      such that
 there are   unique      eigenvectors $ h_V  \in   C_{\wwww_m}(H)$ and $ \mu_V \in \PP_{\mm_\gamma}(H)$ for the semigroups $\PPPP_t^V$ and $\PPPP_t^{V*}$ corresponding to an
      eigenvalue $\la_V>0$, i.e,
 $$\PPPP^{V*}_t\mu_V=\la_V^t \mu_V, \quad\quad \PPPP^V_t h_V=\la_V^t h_V\quad  \quad \text{for $t>0$},
$$     and  normalised by     $\langle h_V,\mu_V\rangle=1$.
 For   any   $\psi\in C_{\mm_\gamma} (H),\,     \nu\in \PP(H)$, and~$R>0$,   the following limits hold as~$t\to+\ty$:
\begin{align*}
\lambda_V^{-t}\PPPP_t^V\psi\to\,&\lag \psi,\mu_V\rag h_V
\,\,\textup{in } C_b(B_H(R))\cap L^1(H,\mu_V),  \\
\lambda^{-t}_V\PPPP_t^{V*}\nu \to\,&\lag h_V,\nu\rag\mu_V
 \,\,\textup{in }\MM_+(H).
\end{align*}
Furthermore, for any   $M>0$ and $\varkappa\in (0,\gamma)$,
$$
\lambda_V^{-t}\,\E_\nu\left\{\exp\bigg(\int_0^tV(u_s)\dd s \bigg) \psi(u_t)\right\}\to\langle \psi,\mu_V\rangle\,\langle h_V,\nu\rangle
$$
  uniformly w.r.t. $\nu\in\Lambda(\varkappa,M)$ as  $t\to+\ty$.
\end{theorem}
 This theorem   improves Theorem~1.1 in~\cite{VN-2019} in two directions. First, in this theorem,   the noise is very degenerate, while in \cite{VN-2019}  all the Fourier modes are assumed to be directly perturbed by the noise. Second, in the present situation,  the class of functions~$V$ is   larger, since the result   in \cite{VN-2019}  applies only to functions depending  on finite-dimensional projection of~$u$.

Theorem~\ref{T:1.2}    can be viewed as    an improvement of Theorem~2.1 in~\cite{HM-2006}. Indeed, in the case~$V=0$, the Feynman--Kac semigroup reduces to the   Markov semigroup with eigenvalue $\la_V=1$, eigenvector $h_V=\mek$ (the function identically equal to~$1$~on~$H$), and the measure~$\mu_V =\mu$ is the unique stationary measure.  The above limits   imply that $\mu$ is mixing.

Theorem~\ref{T:1.1} is derived from Theorem~\ref{T:1.2} by using a Kifer type criterion
in unbounded spaces. Since this derivation is literally the same as in the non-degenerate case (see Section~1 in~\cite{VN-2019}), we do not give the details. The proof
of Theorem~\ref{T:1.2} is discussed in the next subsection.

\subsection{Proof of  Theorem~\ref{T:1.2}}\label{S:1.2}

   The proof of Theorem~\ref{T:1.2} is carried out  by applying     a result on   large-time asymptotics of generalised Markov semigroups
    established in \cite{JNPS-2018} in the discrete-time setting and extended  in \cite{MN-2015} to the continuous-time.~Here we apply that result to the
    Feynman--Kac semigroup $\PPPP_t^V$ and the associated   kernel~$P_t^V(u,\Gamma)=(\PPPP_t^{V*} \de_u ) (\Gamma)$,~$u\in H$, $     \Gamma\in  \BBBBB(H),$
  where
   $\delta_u$ is the Dirac measure concentrated at $u$.

By the  regularising property of the NS system,
the measure~$P^V_{t}(u,\cdot)$ is concentrated on the space~$H^2$ for any  $u\in H$ and   $t>0$.~For any~$R>0$, let us
  denote~$X_R=B_{H^2}(R)$, and let $V\in C_b(H)$ be arbitrary.~Then  the following   properties~hold.
   \begin{description}
       \item[Growth properties.]        There are     numbers $R_0>0$, $\gamma\in (0,\gamma_0)$, and $m\ge1$    such that the following quantities are finite:
\begin{equation} \label{1.3}
 \sup_{t\ge0}
\frac{\|\PPPP^V_t\wwww_m\|_{L_{\wwww_m}^\ty}}{\|\PPPP^V_t{\mek}\|_{R_0}},\,\,\,\,
 \sup_{t\ge0}
\frac{\|\PPPP_t^V\mm_\gamma\|_{L^\ty_{\mm_\gamma}}}{\|\PPPP_t^V{\mek}\|_{R_0}},  \,\,\, \,
   \sup_{t\ge1}
\frac{\|\PPPP_t^V \Phi  \|_{L^\ty_{\mm_\gamma}}}{\|\PPPP_t^V{\mek}\|_{R_0}},
\end{equation}
where~$\|\psi\|_R=\sup_{u\in X_R}|\psi(u)|$  and~$\Phi(u)=\|u\|_{H^2}^2$.
   \item[Existence of an eigenvector.]For any  $t>0$, there is a measure $\mu_{t,V}\in\PP(H)$ and a number $\lambda_{t,V}>0$ such that  $\PPPP_t^{V*}\mu_{t,V}=\lambda_{t,V}\mu_{t,V}$.~Moreover,  for any numbers~$\varkappa\in (0,\gamma_0)$     and  $n,m\ge1$, we have
  \begin{gather*}
  \int_H\!  \left(\|u\|_{H^2}^{n}+\mm_\varkappa (u) \right) \mu_{t,V}(\dd u)<+\ty,  \\
 \|\PPPP_t^V\wwww_m\|_{X_R}\int_{X_R^c}\!\wwww_m(u)\mu_{t,V}(\dd u)\to0\,\,\textup{as } R\to+\ty.
\end{gather*}
    \item[Time-continuity.]
    For any 	 $m\ge1$,  $ \psi\in C_{\wwww_m}(H) $, and  $u\in H$,
    the function $t\mapsto\PPPP^V_t \psi(u) $, $\R_+\to\R$ is continuous.
    \item[Uniform irreducibility.] For
  any $\rho, r, R> 0$, there are numbers       $l=l(\rho,r,R)>0$ and~$p=p(V,\rho,r)>0$  such that
\begin{equation}\label{1.4}
P^V_l(u_0,B_{H}(\hat u,r))\ge p
\end{equation} for any  $u_0\in X_R$ and  $\hat u\in X_\rho$.
\item[Uniform Feller property.]
   The family of functions $\{\|\PPPP_t^V\mek\|_R^{-1} \PPPP_t^V\psi,t\geq 0 \}$
is~uniformly equicontinuous\,\footnote{By   uniform equicontinuity of  $\{\|\PPPP_t^V{\mek}\|_R^{-1}\PPPP^V_t \psi, t\geq 0\}$ on  $X_R$ we mean  that  for any $\eps>0$,  there  is $\delta>0$ such that
$\|\PPPP^V_t{\mek}\|_R^{-1} \big|\PPPP_t^V\psi (u)-\PPPP_t^V\psi (u')\big|<\e$ for any  $u,u'\in X_R$ with  $\|u-u'\|<  \delta $ and any $t\ge0$.
} on    $X_R$ for any $V,  \psi\in C_b^1(H)$ and $R\geq R_0.$
\end{description}
The first three of the above properties   are established\,\footnote{These propositions and lemma in~\cite{VN-2019} are formulated in the case when $X_R=B_{H^1}(R)$  and   the noise is non-degenerate. However, their proofs work in the setting
 of the present paper   without any change.} in Propositions~2.1 and~2.5 and Lemma~2.3 in~\cite{VN-2019}.~The proof of the uniform irreducibility is given below, and the   uniform Feller property is established in Section \ref{S:3}.~Theorem~\ref{T:1.2} is obtained by  applying Theorem~7.4 in~\cite{MN-2015} and by literally repeating the arguments of Section~4 in~\cite{VN-2019}.

    \begin{proof}[Proof of  uniform irreducibility] Let $P_t(u_0,\cdot)$  be   the Markov transition kernel of the family~$(u_t,\IP_{u_0})$. The boundedness of
      $V$ implies that
\begin{equation}\label{1.5}
P_t^V(u_0,\dd v)\ge e^{-t\|V\|_\infty}P_t(u_0,\dd v) \quad\textup{for } t>0,\, u_0\in H.
\end{equation}  According to~\cite{AS-2005, AS-2006}, under Condition~\hyperlink{H}{\rm(H)}, the NS system is approximately controllable in the space~$H$ by controls  taking values in the space
$$
\HH_\KK=\lspan\{e_l:l\in \KK\}.
$$ This implies  that, for any~$u_0, \hat u\in H$ and $r>0$, there is a function $\zeta\in C^\ty([0,1];  \HH_\KK)$ such~that
$$
\|u(1,u_0, \zeta)-\hat u\| <  r,
$$where $u(t,u_0, \zeta)$ is the solution of the deterministic NS system \eqref{0.3}  with the initial condition $u(0)=u_0$ and the (control) force~$\eta= \p_t \zeta$.~Using the fact that the mapping $(u_0, \zeta)\mapsto u(1,u_0, \zeta)$ is continuous from $H\times C([0,1]; \HH_\KK)$ to~$H$, the non-degeneracy of the law of the Wiener process~$ W$ in $C([0,1]; \HH_\KK)$ (i.e.,~the~support of the law of~$ W$ coincides with the
entire~space $C([0,1]; \HH_\KK)$), a simple compactness argument, and inequality \eqref{1.5}, we~arrive~at~\eqref{1.4}.
  \end{proof}

  \section{Elements of Malliavin calculus}\label{S:2}

 The uniform Feller property is proved by using   Malliavin calculus analysis from the papers~\cite{MP-2006, HM-2006, HM-2011}.
   In this section,  we recall some basic definitions and estimates from there.  To match the framework of these    papers, we rewrite the NS system \eqref{0.3}  in the vorticity formulation:
   \begin{equation}
  \label{2.1}
    \p_t w -\nu \Delta w+B(\KK w,w)=\sum_{l \in \KK }  b_l |l|^2 \dot  W_l(t) \phi_l,
  \end{equation}where $w=\nabla\wedge u$,   $B(u,w)=\langle u,\nabla \rangle w$,
  and  $\KK$ is the Biot--Savart operator
  $$
   \KK w=\sum_{l\in \Z_*^2} |l|^{-2} l^\bot w_{-l} \phi_l
  $$
 with $|l|^2=l_1^2+l_2^2$, $l^\bot=(-l_2,l_1)$,   $w_l=\lag w,\phi_l\rag$, and
 $$
\phi_l(x)=
\begin{cases}
 \sin \lag l,x\rag& \text{if }l_1>0\text{ or } l_1=0,\, l_2>0, \\  -\cos\lag l,x\rag   & \text{if }l_1<0\text{ or } l_1=0,\, l_2<0, \quad l =(l_1,l_2).
 \end{cases}
$$
  The operator $\KK$ is continuous from $H^s(\T^2;\R)$ to $H^{s+1}(\T^2;\R^2)$ for any~$s\in \R$; it~allows to recover   the velocity field   from the vorticity via $u=\KK w$.

     We consider Eq.~\eqref{2.1} in the space  $\tilde H$  of real-valued square-integrable functions on~$\mathbb{T}^2$ with zero mean value (see \eqref{0.6}); it    is endowed with the $L^2$ norm $ \|\cdot\|.$ Since~the underlying probability space  plays no role,
       without loss of generality,  we~can assume that   $\Omega$ is the  Wiener space, $W(t)=\{W_l(t)\}_{l\in\KK}$ is the canonical process, and $\pP$ is the Wiener measure.~Furthermore,  we   denote by $\{\theta_l\}_{l\in\KK}$ the standard basis in  $\R^d$ with $d=|\KK|$, and define a
  linear map $Q:\R^d\to  \tilde H$   by~$Q\theta_l=b_l |l|^2\phi_l.$
Let~$w_t=\Phi(t,w,W_\cdot)$ be~the solution of  Eq.~\eqref{2.1} with initial value $w(0)=w\in \tilde{H}$.
For any $0\le s\le t$ and $\xi\in \tilde H$, let~$J_{s,t}\xi$  be the   solution of the linearised problem:
\begin{align}
	\partial_t  J_{s,t}\xi-\nu \Delta J_{s,t}\xi +\tilde{B}(w_t,J_{s,t}\xi)&=0, \label{2.2}
  \\  J_{s,s}\xi&=\xi, \nonumber
\end{align}
where $\tilde{B}(w,v)=B(\KK w,v)+B(\KK v,w)$.

Recall that, for given   $T>0$ and $v\in L^2([0,T];\R^d),$
the Malliavin derivative of~$w_t$ in the direction~$v$ is defined by
$$
\cD^v w_t=\lim_{\eps \to  0}\frac{1}{\e}
  \left(\Phi(t,w_0,W+\eps \int_0^\cdot v\dd  s)-\Phi(t,w_0,W)\right), \quad
$$
  where the limit   holds almost surely (e.g., see the book~\cite{nualart2006} for finite-dimensional setting or the papers~\cite{MP-2006, HM-2006, HM-2011, FGRT-2015}   for   Hilbert space~case).~By the Riesz representation theorem, there is a linear operator $  \cD:L^2(\Omega,  \tilde H)\to L^2(\Omega; L^2([0,T];\R^d)\otimes  \tilde H)$ such that
\begin{equation}\label{2.3}
	  \cD^v w =\lag    \cD w,v \rag_{L^2([0,T];\R^d)}.
\end{equation}
On the other hand,  we have
\begin{equation}\label{2.4}
	\cD^v w_t=\cA_{0,t}v,
\end{equation} where   $\cA_{s,t}:L^2([s,t];\R^{d})\to  \tilde H$  is  the random operator defined  by
\begin{equation}
\label{2.5}
\cA_{s,t}v=\int_{s}^{t}J_{r,t}Qv(r)\dd  r, \quad 0\le s\le t\le T,
\end{equation}i.e.,  $\cA_{s,t}v$ is
  the solution of the linearised problem with a source term:
\begin{align*}
	\partial_t  \cA_{s,t}v-\nu \Delta \cA_{s,t}v +\tilde{B}(w_t,\cA_{s,t}v)&=Qv,
  \\  \cA_{s,s}v&=0.
\end{align*}
The adjoint $\cA_{s,t}^*:\tilde H \to  L^2([s,t];\R^{d})$      is given by
$$
(\cA_{s,t}^*\xi)(r)=Q^{*}J_{r,t}^*\xi, \quad  \xi\in \tilde H,\, r\in [s,t],
$$
where $Q^*:\tilde H\rightarrow \R^{d}$ is the adjoint of $Q$.

Let us
 denote by $J^{(2)}_{s,t}(\phi,\psi)$ the second derivative of $w_t$ with respect to $w$ in the directions of $\phi$ and $\psi$. It is the solution of the problem
\begin{align*}
   \partial_t J^{(2)}_{s,t}(\phi,\psi)- \nu \Delta J^{(2)}_{s,t}(\phi,\psi)+\tilde{B}(J_{s,t}\phi,J_{s,t}\psi)
  +\tilde{B}(w_t,J^{(2)}_{s,t}(\phi,\psi))&=0,
  \\  J^{(2)}_{s,s}(\phi,\psi)&=0.
\end{align*}The next lemma follows from Lemma 4.10 in \cite{HM-2006}.
 \begin{lemma}
\label{L:2.1}
  For any $\kappa,p>0$, $0\le  \tau<   T$, and $w\in \tilde H$, we have
  \begin{align}
    \label{2.6}
    \E_w \sup_{s<t\in  [\tau,T]}\|J_{s,t}\|^p_{\cL( \tilde H,\tilde{H})}&\le    C   \exp\{\kappa\|w\|^2\},
    \\
\nonumber
    \E_w   \sup_{s<t\in [\tau,T]} ||| J^{(2)}_{s,t} |||^p &\le    C   \exp\{\kappa\|w\|^2\} ,
  \end{align}
  where $ ||| J^{(2)}_{s,t} |||=\sup_{\|\phi\|,\|\psi\|\le1}  
   \|J^{(2)}_{s,t}(\phi,\psi)\|$, and $C=C(\kappa, p,T-\tau,  \mathfrak{B}_0)>0 $ is a constant.
  \end{lemma}

For any  $0\le  s< t,$     the Malliavin operator    is defined by
$$
\cM_{s,t}=\cA_{s,t}\cA_{s,t}^*:\tilde{H}\rightarrow \tilde{H}.
$$ It is a
  non-negative self-adjoint operator, so its regularisation~$\cM_{s,t}+\beta\I$
is invertible for any  $\beta>0$. Here  $\I$ is the identity.
The following lemma gathers some estimates from Section~4.8 in~\cite{HM-2006}  and Lemma~A.6 in~\cite{FGRT-2015}.
\begin{lemma}
\label{L:2.2}There is a constant $C=C(\mathfrak{B}_0)>0$ such that, for any
  $0\le s<t$, $\beta>0$, and~$w\in \tilde H$, we have
  \begin{gather}
   \label{2.7}  \|\cA_{s,t}\|^2_{\cL(L^2([s,t];\R^d),\tilde{H})} \le  C \int_s^t\|J_{r,t}\|_{\cL( \tilde H,\tilde{H})}^2\dd  r ,
     \\   \|\cA_{s,t}^*(\cM_{s,t}+\beta\I)^{-1/2}\|_{\cL(\tilde{H},L^2([s,t];\R^d))} \le  1,\label{2.8}
   \\
   \|(\cM_{s,t}+\beta\I)^{-1/2}\cA_{s,t}\|_{\cL(L^2([s,t];\R^d),\tilde{H})} \le  1,\label{2.9}
   \\  \|(\cM_{s,t}+\beta\I)^{-1/2}\|_{\cL(\tilde{H},\tilde{H})} \le  \beta^{-1/2}.\label{2.10}
  \end{gather}
\end{lemma}
We shall use the notation
$$
  \cD_r F=(\cD F) (r), \quad \cD^j F=(\cD F)^j,\quad \cD_r^j F=(\cD F)^j(r), \quad j=1,\ldots, d.
$$
From the equalities  \eqref{2.3}-\eqref{2.5}   it follows that  $\cD_r^iw_t=J_{r,t}Q\theta_i$, $0\le r\le t$.
From this   and  \eqref{2.2}, we conclude that, for $0\le s<t$,
$$
   \partial_t\cD_r^iJ_{s,t}\xi
   -\nu \Delta\cD_r^iJ_{s,t}\xi+\tilde{B}(w_t, \cD_r^iJ_{s,t}\xi)
  +\tilde{B}( J_{r,t}Q\theta_i,J_{s,t}\xi)=0.
$$
Furthermore, by  the variation of constants formula,   we have
 $$
 \cD_r^iJ_{s,t}\xi=
\begin{cases}
 J^{(2)}_{r,t}( Q\theta_i,J_{s,r}\xi) &  \text{ for } r\geq s, \\  J^{(2)}_{s,t}(J_{r,s}Q\theta_i,\xi)   &   \text{ for } r\le  s.
 \end{cases}
$$
 This equality and Lemma \ref{L:2.1}  imply the following lemma.  For further details, we~refer the reader to  Section~4.8 in~\cite{HM-2006} and Lemma~A.7 in \cite{FGRT-2015}.
\begin{lemma}\label{L:2.3}
The   operators $J_{s,t}$, $\cA_{s,t}$, and $\cA_{s,t}^*$
are Malliavin differentiable, and for any
$\kappa>0$, $r\in[s,t]$,  $p> 0$, and $w\in \tilde H$, the following inequalities~hold
\begin{align}
  \label{2.11}  \E_w  \|\cD_r^iJ_{s,t}\|^p_{\cL(\tilde{H},\tilde{H})} & \le    C \exp\{\kappa \|w\|^2\},
  \\  \label {2.12}
   \E_w  \|\cD_r^i \cA_{s,t}\|^p_{\cL(L^2([s,t];\R^d),\tilde{H})} & \le    C \exp\{\kappa \|w\|^2\},
  \\   \label{2.13}  \E_w  \|\cD_r^i\cA^*_{s,t}\|^p_{\cL(\tilde{H},L^2([s,t];\R^d))} & \le    C \exp\{\kappa \|w\|^2\},
\end{align}
where $C=C(\kappa,p,t-s,\mathfrak{B}_0)>0$.
\end{lemma}

\section{Proof of  uniform Feller property} \label{S:3}

\subsection{Reduction to a gradient estimate} \label{S:3.1}

 The aim of this  section is to  prove  the following proposition.
\begin{proposition}\label{P:3.1}
Under Condition~\hyperlink{H}{\rm(H)},  for any $V,  \psi\in C_b^1(H)$, there is a number $ R_0=R_0(V)>0$ such that
 the family   $\left\{\|\PPPP_t^V\mek\|_R^{-1} \PPPP_t^V\psi,t\geq 0 \right\}$
is uniformly equicontinuous  on    $X_R$ for any  $R\geq R_0.$
\end{proposition}
\begin{proof} For any $V,\psi\in C_b^1(H),$ let us define  functions
 $\tilde V, \tilde \psi \in C_b^1(\tilde H)$
 by
  $\tilde{V}(w)= V(\KK w)$
 and $ \tilde{\psi}(w)= \psi(\KK w)$, $ w\in\tilde H$.
The     Feynman--Kac semigroup  associated with  Eq.~\eqref{2.1}  is~given~by
$$\tilde \PPPP_t^{\tilde V} \tilde{\psi}(w)=\E_w \left\{ \exp\left(\int_0^t\tilde V (w_s) \dd s\right) \tilde \psi (w_t)\right\}, \quad \tilde \PPPP_t^{\tilde V}: C_b^1(\tilde H )\to C_b^1(\tilde H).
$$In what follows,
the number $R_0$ is chosen such that the growth properties~\eqref{1.3} hold.
 In  the next   subsection, we   prove the following proposition  (cf.~Proposition~4.3 in~\cite{HM-2006}).
\begin{proposition}\label{P:3.2}
Under the conditions of Proposition \ref{P:3.1},
  for any numbers  $ \kappa>0$ and $ a\in (0,1)$, there is a constant  $C=C(\kappa,a,\|\nabla V\|_\ty,\|V\|_\ty )>0$ such that
 \begin{equation}\label{3.1}
 \|\nabla_\xi \tilde{\PPPP}_t^{\tilde V}\tilde \psi(w)\|
 \le
  C \exp\{\kappa \|w\|^2\}\|\PPPP_t^V\mek\|_{R_0} \left[
 \|\nabla \psi\|_\ty a^t+\|\psi\|_\ty \right]\|\xi\|
 \end{equation}
 for any  $w,\xi\in \tilde H$ and  $ t\geq 0$.  Here
 $\nabla_\xi$  is the derivative with respect
to the initial condition in the direction $\xi$.
\end{proposition}This  result implies Proposition \ref{P:3.1}. Indeed, let us take any~$u_1,u_2\in X_R$ and set~$w_i=\nabla\wedge u_i$,  $i=1,2$.
Using inequality \eqref{3.1} with any~$\kappa>0$ and $a\in (0,1)$
  and an interpolation inequality, we see that
\begin{align*}
 \big| \PPPP_t^V\psi(u_1)-\PPPP_t^V\psi(u_2)\big|
& = \big| \tilde{\PPPP}_t^{\tilde V} \tilde{\psi}(w_1)-
\tilde{\PPPP}_t^{\tilde V}\tilde{\psi}(w_2)\big|
 \\ & \le  C  \|\PPPP_t^V\textbf{1}\|_{R_0} \|w_1-w_2\|
\\ & \le  C \|\PPPP_t^V\textbf{1}\|_{R_0} \|u_1-u_2\|_1
 \\ & \le  C \|\PPPP_t^V\textbf{1}\|_{R_0} \|u_1-u_2\|^{1/2}\|u_1-u_2\|_2^{1/2}
\\ &\le  C \|\PPPP_t^V\textbf{1}\|_{R_0} \|u_1-u_2\|^{1/2},
\end{align*}
where    $C=C(R,\|V\|_\ty,\|\nabla V\|_\ty,\|\psi\|_\ty,\|\nabla \psi\|_\ty)>0.$
This~completes the proof of  Proposition \ref{P:3.1}.
\end{proof}

\subsection{Proof of Proposition \ref{P:3.2}}\label{S:3.2}

Let us take any  $\xi\in \tilde{H}$ with $\|\xi\|=1$, denote
$$
\Xi_t    = \exp\left(\int_0^t \tilde V(w_s) \dd s\right),
$$and compute the derivative of $\tilde \PPPP_t^{\tilde V} \tilde{\psi}(w)$ with respect to $w$ in the direction $\xi$:
\begin{equation}
   \nabla_\xi \tilde\PPPP_t^{\tilde{V}}\tilde{\psi}(w) =  \E_w\left[\Xi_t \tilde{\psi}( w_t)\int_0^t \nabla \tilde{V}(w_s) J_{0,s}  \xi \dd  s
  +\Xi_t \nabla\tilde{\psi}(w_t) J_{0,t}\xi\right].
 \label{3.2}
\end{equation}
 Inspired by the papers~\cite{HM-2006, HM-2011}, the idea of the proof of   Proposition \ref{P:3.2} is~to approximate the perturbation
  $J_{0,t}\xi$   caused by the perturbation  $\xi$ of the initial condition with a variation $\cA_{0,t}v$ coming from a variation   of the noise      by an appropriate   process $v$. Let us denote by $\rho_t$ the  residual error between~$J_{0,t}\xi$    and~$\cA_{0,t}v$:
$$
  \rho_t=J_{0,t}\xi-\cA_{0,t}v,
$$ replace the term $J_{0,t}\xi$ in \eqref{3.2}  by $\cA_{0,t}v+   \rho_t$, and recall that $\cA_{0,t}v=\cD^vw_t$ is the Malliavin derivative of~$w_t$ in the direction~$v$. Then, at least formally, using the Malliavin
chain rule (see Proposition~1.2.3 in~\cite{nualart2006}), we have
\begin{align}
  \nonumber  \nabla_\xi \tilde\PPPP_t^{\tilde{V}}\tilde{\psi}(w)    & =  \E_w\left[\Xi_t \tilde{\psi}(w_t)\int_0^t \nabla \tilde{V}(w_s)\cD^vw_s  \dd  s
  +\Xi_t \nabla \tilde{\psi}(w_t)\cD^vw_t \right]
 \\    \nonumber  &\quad + \E_w\left[\Xi_t\tilde{\psi}(w_t)\int_0^t \nabla \tilde{V}(w_s)\rho_s  \dd  s
 +\Xi_t\nabla \tilde{\psi}(w_t)\rho_t \right]
 \\    \nonumber   &=\E_w \left[\cD^v \left(\Xi_t\tilde{\psi}(w_t)\right)\right]
 + \E_w\left[\Xi_t\tilde{\psi}(w_t)\int_0^t \nabla \tilde{V}(w_s)\rho_s  \dd  s\right]
 \\       &\quad +\E_w\left[\Xi_t\nabla \tilde{\psi}(w_t)\rho_t \right]
 =I_1+I_2+I_3. \label{3.3}
\end{align}The term $I_1$ is treated using Malliavin integration by parts formula
  (see Lemma~1.2.1 in \cite{nualart2006}):
  \begin{equation}\label{3.4}
   I_1   = \E_w \left[\Xi_t\tilde{\psi}(w_t) \int_0^t v(s)\dd W(s) \right],
\end{equation}where  the  stochastic integral $ \int_0^t v(s)\dd W(s) $ is   in the Skorokhod sense.~The goal is  to choose the process~$v$ in a such way that   the terms~$I_i$, $i=1,2,3$ are bounded by the right-hand side of inequality~\eqref{3.1}.~We use the same choice of~$v$ as in the papers~\cite{HM-2006, HM-2011}.~More precisely, for any integer $n\ge0$,   the~restriction~$ v_{n,n+1}$ of the process $v$  to the time interval $[n,n+1]$   is defined by
  \begin{equation}
\label{3.5}
 v_{n,n+1}(t)=\begin{cases}  \big(\cA_{n,n+1/2}^*\left(\cM_{n,n+1/2}+\beta \I\right)^{-1}
  J_{n,n+1/2}\rho_n\big)(t),  \,\,\,\,\quad t \in [n,n+1/2], \\ 0,\qquad \qquad \qquad \qquad \qquad \qquad \qquad \qquad \qquad  \quad    t\in [n+1/2,n+1], \end{cases}	
\end{equation}where  we set $\rho_0=\xi$ and   $\beta>0$ is a small parameter. This choice allows to have an exponential decay for the moments of $\rho_t$ and of the Skorokhod integral as proved in the following   lemmas. Inequality~\eqref{3.1} is proved by combining these lemmas (with an appropriate choice of parameters therein) and using a  growth property of the Feynman--Kac semigroup.

The following two lemmas are
  versions of Propositions~4.13 and~4.14 in~\cite{HM-2006}.
Since their formulations   differ from the original ones, we give rather  detailed~proofs based on the estimates recalled in Section \ref{S:2}.
\begin{lemma}
\label{L:3.3}
  For any $\kappa>0$ and $\alpha>0$,  there are    constants  $\beta=\beta(\kappa,\alpha)>0$  and~$C=C(\kappa,\alpha)>0$  such that
  \begin{equation}
  \label{3.6}
    \E\|\rho_t\|^4 \le  C  \exp\left\{\kappa \|w\|^2-\alpha t \right\}\quad \text{for any $w\in \tilde H$ and $t\ge0$.}
  \end{equation}
\end{lemma}
\begin{proof}
For  integer times, this result  is established  in   Proposition~4.13 in~\cite{HM-2006} (this is   where Condition~\hyperlink{H}{\rm(H)} is used).  Therefore,  there are    $\beta=\beta(\kappa,\alpha )>0$  and~$C=C(\kappa,\alpha )>0$  such that
  \begin{equation}
 \label{3.7}
    \E\|\rho_n\|^4 \le  C \exp\{\kappa \|w\|^2- \alpha  n\} \quad \text{for any $w\in \tilde H$ and $n\ge0$}.
  \end{equation}From the construction it follows that
 $$
 \rho_t=\begin{cases} J_{n,t}\rho_n-\cA_{n,t}v_{n,t}, & \text{for }t \in [n,n+1/2], \\ J_{n+1/2,t}\rho_{n+1/2}, & \text{for }t\in [n+1/2,n+1] \end{cases}	
 $$ for any  $n\ge0$.
  Using \eqref{3.5} and inequalities \eqref{2.8} and \eqref{2.10}, we get
    \begin{equation} \label{3.8}
    \|v_{n,n+1/2}\|_{L^2([n,n+1/2];\R^d)} \le \beta^{-1/2}\|J_{n,n+1/2} \rho_n\|.
\end{equation}
Hence,   for any $t\in [n,n+1/2]$,
   \begin{align}
    \|\rho_t\|&\le  \|J_{n,t}\rho_n\|+\|\aA_{n,t}v_{n,t}\|  \nonumber
    \\  \nonumber  &\le  \|J_{n,t}\rho_n\|+\|\aA_{n,t}\|_{\cL(L^2([n,t];\R^d),\tilde H)} \|v_{n,t}\|_{L^2([n,n+1/2];\R^d)}
    \\  \nonumber  &\le  \|J_{n,t}\rho_n\|+\|\aA_{n,t}\|_{\cL(L^2([n,t];\R^d),\tilde H)} \|v_{n,n+1/2}\|_{L^2([n,n+1/2];\R^d)}
    \\ \label{3.9}
     &\le C\Big(\|J_{n,t}\rho_n\| +\beta^{-1/2}  \|J_{n,n+1/2}\rho_n\| \sup_{s\in [n,t]}\|J_{s,t} \|_{\cL( \tilde H,\tilde{H})} \Big),
  \end{align}
  where we used  \eqref{2.7} and \eqref{3.8}. For any
     $t\in [n+1/2,n+1],$ it holds that
$$
\|\rho_t\|\le  \sup_{s\in [n+1/2,t]}\|J_{s,t}\rho_{n+1/2}\| .
$$
  Combining this with  inequalities    \eqref{2.6},   \eqref{3.7},~\eqref{3.9}, the Cauchy--Schwarz inequality, and the fact that  $\kappa>0$ and $\alpha>0$ are arbitrary,   we arrive at~\eqref{3.6}.
\end{proof}

\begin{lemma}
\label{L:3.4}The constants $\beta>0$ and $C>0$ in Lemma \ref{L:3.3} can be chosen such that     also
  \begin{equation}
  \label{3.10}
  \E \left| \int_{n}^t v(s)\dd W(s)\right|^2 \le   C\exp\{\kappa \|w\|^2-\alpha  n \}
  \end{equation}
    for any $n\geq 0$,   $t\in [n,n+1]$, and $w\in \tilde H$.
\end{lemma}
\begin{proof}  In this proof, we consider the endpoint case $t=n+1$; the case $t\in [n,n+1)$ is treated in a similar way.
  Using  the generalised It$\hat{\text{o}}$ isometry (see Section~1.3 in~\cite{nualart2006}) and the fact that $v(t)=0$ for $t\in [n+1/2,n+1]$ (see \eqref{3.5}), we~obtain
  \begin{align}
  \E \left| \int_{n}^{n+1} v(s)\dd  W(s)\right|^2
  & =  \E \int_{n}^{n+1/2} |v(s)|_{\R^d}^2 \dd  s   \nonumber\\&\quad+\E\int_n^{n+1/2}\!\!\!\int_{n}^{n+1/2} \text{Tr}(\cD_sv(r)\cD_rv(s))\dd  s \dd  r\nonumber
   \\  \nonumber   & \le    \E \int_{n}^{n+1/2} |v(s)|_{\R^d}^2 \dd  s \nonumber\\&\quad+\E \int_{n}^{n+1/2}\!\!\!\int_{n}^{n+1/2} |\cD_rv_{n,n+1/2}(s)|_{\R^d\times \R^d}^2\dd  s\dd  r\nonumber\\  & =L_1+L_2.
  \label{3.11}
  \end{align}We estimate $L_1$ by using \eqref{2.6},
       \eqref{3.7},  and \eqref{3.8}:
     \begin{align}
  \nonumber  \E \int_{n}^{n+1/2} |v(s)|_{\R^d}^2 \dd  s &\le     \beta^{-1}\E \|J_{n,n+1/2}\rho_{n}\|^2 \\
  &\le \nonumber
  C \beta^{-1}  \exp\{\kappa \|w\|^2/2 \} \left(\E \|\rho_n\|^4\right)^{1/2}
   \\
     \label{3.12}
   &\le   C\exp\{\kappa \|w\|^2-\alpha  n/2 \}.
  \end{align}
   To estimate $L_2$,   we use the explicit form of   $\cD_r v.$
  Notice  that,   for any $r\in[n,n+1/2]$ and~$i=1,\dots, d$,
\begin{align*}
   \cD_r^iv_{n,n+1/2}
 &= \cD_r^i(\cA_{n,n+1/2}^*) (\cM_{n,n+1/2}+\beta \I)^{-1}J_{n,n+1/2}\rho_{n}
    \\ &\quad + \cA_{n,n+1/2}^*(\cM_{n,n+1/2}+\beta\I)^{-1}
 \\ & \quad\quad\times \Big( \cD_r^i(\cA_{n,n+1/2})\cA_{n,n+1/2}^*
  +\cA_{n,n+1/2}\cD_r^i(\cA_{n,n+1/2}^*)\Big) \\ & \quad\quad\times
 (\cM_{n,n+1/2}+\beta\I)^{-1}
 J_{n,n+1/2} \rho_{n}\\ & \quad +\cA_{n,n+1/2}^* (\cM_{n,n+1/2}+\beta\I)^{-1}\cD_r^i(J_{n,n+1/2}) \rho_{n}.
\end{align*}
By inequalities \eqref{2.8}-\eqref{2.10}, we have
\begin{align*}
   \|\cD_r^iv_{n,n+1/2}\|_{L^2([n,n+1/2];\R^d)}
  & \le   \beta^{-1}\| \cD_r^i(\cA_{n,n+1/2})\|_{\cL(L^2([n,n+1/2];\R^d),\tilde{H})} \\&\quad\quad\times
  \|J_{n,n+1/2} \rho_{n}\|
    \\ &\quad +2\beta^{-1}\| \cD_r^i(\cA_{n,n+1/2}^*)\|_{\cL(\tilde{H},L^2([n,n+1/2];\R^d))} \\&\quad\quad\times
  \|J_{n,n+1/2} \rho_{n}\|  \\ &\quad
  + \beta^{-1/2} \| \cD_r^i(J_{n,n+1/2})  \rho_{n}\|.
\end{align*}
Inequalities \eqref{2.6}, \eqref{2.11}-\eqref{2.13}, and \eqref{3.7}, imply  that
\begin{align}
  \E \int_{n}^{n+1/2}\!\!\!\int_{n}^{n+1/2} |\cD_rv_{n,n+1/2}(s)|_{\R^d\times \R^d}^2\dd  s\dd  r
   &  \le  C \beta^{-2}  \exp\{\kappa \|w\|^2/2 \} \left(\E \|\rho_n\|^4\right)^{1/2}
 \nonumber \\ \label{3.13} &\le   C  \exp\{\kappa \|w\|^2-n\alpha/2 \}.
\end{align}
Combining   estimates \eqref{3.11}-\eqref{3.13} and using the fact that $\kappa>0$ and $\alpha>0$  are~arbitrary,
we obtain the desired result.
\end{proof}
Finally, we will use a growth estimate for the Feynman--Kac semigroup $ \tilde \PPPP^{\tilde V}_t$. From  the first growth esimate in~\eqref{1.3} for the semigroup $  \PPPP^V_t$ it follows that there are numbers
$R_0>0$, $\gamma\in (0,\gamma_0)$, and $m\ge1$
  such that
\begin{equation}\label{3.14}
 \tilde \PPPP^{\tilde V}_t \mek (w)\le  C\, \wwww_m(\KK w)    \|\PPPP^V_t{\mek}\|_{R_0}\quad \text{for any $w\in \tilde H$ and $t\ge0$}.
\end{equation}

Now we are in a position to prove Proposition  \ref{P:3.2}.
\begin{proof}[Proof of Proposition  \ref{P:3.2}]
Replacing $V$ by $V-\inf_{v\in H} V(v),$ without loss of generality, we can assume that $V\ge0$.  Let $v$ be the process defined by \eqref{3.5}, let  $\kappa$ and~$\alpha $ be   positive numbers (to be   chosen later),  and  let the number $\beta=\beta(\kappa,\alpha )>0$  be such that   inequalities~\eqref{3.6} and \eqref{3.10} hold. Furthermore,  let the positive numbers $R_0$ and $m$ be such that inequality  \eqref{3.14} holds.
Then the computations in \eqref{3.3} are rigorously justified, and we need to   estimate the terms $I_1$, $I_2$, and~$I_3$.

{\it Step~1: Estimate for $I_1$}. We write the Skorokhod integral in the term $I_1$  (see~\eqref{3.4})
  as follows
$$
\int_0^t v(s)\dd  W(s) =\sum_{n=1}^{\lfloor t \rfloor}
  \int_{n-1}^{n} v(s)\dd  W(s) + \int_{\lfloor t \rfloor}^{t} v(s)\dd  W(s),
$$
where  $\lfloor t \rfloor $  is the largest number less than or equal to $t$ and the sum in the right-hand side is replaced by zero if $t<1$.
Since $v(s)$ is $\cF_{n}$-measurable for~$s\in [n-1,n]$,  the Skorokhod integral  $ \int_{n-1}^{n} v(s)\dd  W(s)$ is also $\cF_{n}$-measurable. Hence, using the Markov property, we obtain
\begin{align*}
 \nonumber   I_{1,n}  =&\,\, \E_w \left[\Xi_t\tilde{\psi}(w_t)   \int_{n-1}^{n} v(s)\dd  W(s) \right]
  \\   \nonumber   =& \,\,
  \E_w \left[\E_w \left(\Xi_t\tilde{\psi}(w_t) \int_{n-1}^{n} v(s)\dd  W(s)\Big|\cF_{n}  \right)\right]
  \\ \nonumber  =&\,\,
  \E_w \left[\Xi_n \int_{n-1}^{n} v(s)\dd  W(s) \,\E_w \left( \exp\left\{\int_{n}^t\tilde{V}(w_s)\dd  s \right\}\tilde{\psi}(w_t)\Big|\cF_{n}\right)   \right]\\   =& \,  \,
  \E_w \left[\Xi_n \int_{n-1}^{n} v(s)\dd  W(s) \, \left(\tilde \PPPP^{\tilde V}_{t-n} \tilde \psi\right) (w_n) \right]
\end{align*}
for any $1\le  n\le  \lfloor t \rfloor$.~Using inequalities \eqref{1.2}, \eqref{3.10}, \eqref{3.14},   the assumption that~$V\ge0$,   and the Cauchy--Schwarz inequality, we see that
\begin{align*}
 \nonumber  I_{1,n} &  \le   C \|\psi\|_\ty e^{\|V\|_\ty n}\|\PPPP_t^V\mek\|_{R_0}  \E_w\left[\wwww_{m}(u_{n}) \Big| \int_{n-1}^{n} v(s)\dd  W(s)\Big|\right]
  \\   \nonumber   & \le   C\|\psi\|_\ty   e^{\|V\|_\ty n}\|\PPPP_t^V\mek\|_{R_0} \left(\E_u\wwww_{m}^2(u_{n})\right)^{1/2} \left( \E_w \Big| \int_{n-1}^{n} v(s)\dd  W(s)\Big|^2\right)^{1/2}
    \\    & \le    C \|\psi\|_\ty  e^{\|V\|_\ty n}\|\PPPP_t^V\mek\|_{R_0}   \wwww_{m}(u)  \exp\{(\kappa \|w\|^2 -\alpha  n)/2\},
\end{align*}
where $u=\KK w$ and $u_s=\KK w_s$. Next, using \eqref{3.10} and $V\ge 0$, we get
\begin{align*}
 \nonumber I_{1,\lfloor t \rfloor+1}&=\E_w \left[\Xi_t\tilde{\psi}(w_t) \int_{\lfloor t \rfloor }^{t} v(s)\dd  W(s)  \right]
  \\ &\le  \|\psi\|_\ty
  e^{\|V\|_\ty t }\E_w \left|\int_{\lfloor t \rfloor }^{t} v(s)\dd  W(s)\right|\nonumber\\\  & \le    C\|\psi\|_\ty    e^{\|V\|_\ty t}  \|\PPPP_t^V\mek\|_{R_0}  \exp\{(\kappa \|w\|^2-\alpha  \lfloor t \rfloor )/2 \}.
\end{align*} Combining the estimates for $I_{1,n}$ and $I_{1,\lfloor t \rfloor+1}$, we arrive at
\begin{align*}
I_1\le  C  \|\psi\|_\ty \|\PPPP_t^V\mek\|_{R_0}  \exp\{\kappa \|w\|^2\} \sum_{n=1}^{^{\lfloor t \rfloor}}    \exp\{(\|V\|_\ty -\alpha /2)n\}.
	\end{align*}

{\it Step~2: Estimate for $I_2$}. We first write
\begin{align*}
    I_2 &= \E_w\left[\int_0^{\lfloor t \rfloor } \Xi_t\tilde{\psi}(w_t)\nabla \tilde{V}(w_s)\rho_s  \dd  s\right] +\E_w\left[\int_{\lfloor t \rfloor }^t \Xi_t\tilde{\psi}(w_t)\nabla \tilde{V}(w_s)\rho_s  \dd  s\right]
    \\&= I_{2,1}+I_{2,2}.
\end{align*}Let $\lceil s \rceil$ be the smallest integer greater than or equal to $s$. Then  $\rho(s)$ is   $\cF_{\lceil s\rceil}$-measurable,  and using the Markov property, we obtain
\begin{align*}
    I_{2,1} &=  \E_w\left[ \int_0^{\lfloor t\rfloor}  \E_w \left( \Xi_t\tilde{\psi}(w_t)   \nabla \tilde{V}(w_s)\rho_s \big| \cF_{ \lceil s \rceil } \right)  \dd  s  \right]
  \\ &= \E_w \left[ \int_0^{\lfloor t \rfloor }
  \Xi_{\lceil s \rceil} \nabla \tilde{V}(w_s)\rho_s
  \
  \E_w \left(\exp\left\{\int_{\lceil s \rceil}^t\tilde{V}(w_r)\dd  r\right\}\tilde{\psi}(w_t)   \Big| \cF_{ \lceil s \rceil } \right)  \dd  s  \right]
 \\ &= \E_w \left[ \int_0^{\lfloor t \rfloor }
  \Xi_{\lceil s \rceil} \nabla \tilde{V}(w_s)\rho_s
  \left(
  \tilde \PPPP^{\tilde V}_{t-\lceil s \rceil} \tilde \psi\right) (w_{\lceil s \rceil}) \dd  s  \right].
\end{align*}
Then  inequalities  \eqref{1.2}, \eqref{3.6}, \eqref{3.14},  the assumption that~$V\ge0$,   and the Cauchy--Schwarz inequality imply that
\begin{align*}
   I_{2,1} &  \le
    C \|\psi\|_\ty\|\nabla V\|_\ty \|\PPPP_t^V\textbf{1}\|_{R_0}
    \int_0^{\lfloor t\rfloor}  e^{\|V\|_\ty  \lceil s \rceil  } \E_w  \left[\wwww_m(u_{ \lceil s \rceil  }) \|\rho_s\| \right] \dd  s
    \\ &\le   C
    \|\psi\|_\ty\|\nabla V\|_\ty \|\PPPP_t^V\textbf{1}\|_{R_0}
    \int_0^{\lfloor t\rfloor} e^{\|V\|_\ty  \lceil s \rceil  }  \left(\E_u \wwww_m^2(u_{\lceil s \rceil })\right)^{1/2} \left(\E_w \|\rho_s\|^2\right)^{1/2}  \dd  s
      \\
     & \le
    C\|\psi\|_\ty\|\nabla V\|_\ty \|\PPPP_t^V\textbf{1}\|_{R_0} \wwww_{m}(u)   \exp\{\kappa \|w\|^2/4\}\\&\quad\quad\quad\quad\times
    \int_0^{\lfloor t\rfloor}  \exp\{\|V\|_\ty   \lceil s \rceil -\alpha s/4 \}  \dd  s.
\end{align*}
 To estimate $I_{2,2}$, we use  \eqref{3.6} and $V\ge0$:
\begin{align*}
	I_{2,2}&\le \|\psi\|_\ty e^{\|V\|_\ty t }\|\nabla V\|_\ty \E_w \int_{\lfloor t \rfloor }^t \|\rho_s\| \dd s \\&\le
C  \|\psi\|_\ty e^{\|V\|_\ty t }\|\nabla V\|_\ty\|\PPPP_t^V\mek\|_{R_0}
  \exp\{(\kappa \|w\|^2-\alpha t)/4\}.
\end{align*}Thus
\begin{align*}
I_2 &\le  C\|\psi\|_\ty\|\nabla V\|_\ty \|\PPPP_t^V\textbf{1}\|_{R_0}   \exp\{\kappa \|w\|^2\} \\&\quad\quad\quad\quad\times\left(\exp\{-\alpha t)/4\}+
    \int_0^{\lfloor t\rfloor}   \exp\{\|V\|_\ty   \lceil s \rceil -\alpha s/4 \}  \dd  s\right).
\end{align*}

{\it Step~3: Estimate for $I_3$}.
By  \eqref{3.6}, we have
$$
  |I_3|\le  C \|\nabla \psi\|_\ty  e^{\|V\|_\ty t}\E \|\rho_t\|
  \le  \|\nabla \psi\|_\ty  e^{\|V\|_\ty t}  \exp\{(\kappa \|w\|^2-t\alpha )/4\}.
$$ Choosing
 $
 \alpha \geq 4 \|V\|_\ty-\log a
 $ and  combining the above estimates of the terms $I_i$, $i=1,2,3$ with  \eqref{3.3}, we complete the proof of
Proposition~\ref{P:3.2}.
 \end{proof}

\def\cprime{$'$} \def\cprime{$'$}
  \def\polhk#1{\setbox0=\hbox{#1}{\ooalign{\hidewidth
  \lower1.5ex\hbox{`}\hidewidth\crcr\unhbox0}}}
  \def\polhk#1{\setbox0=\hbox{#1}{\ooalign{\hidewidth
  \lower1.5ex\hbox{`}\hidewidth\crcr\unhbox0}}}
  \def\polhk#1{\setbox0=\hbox{#1}{\ooalign{\hidewidth
  \lower1.5ex\hbox{`}\hidewidth\crcr\unhbox0}}} \def\cprime{$'$}
  \def\polhk#1{\setbox0=\hbox{#1}{\ooalign{\hidewidth
  \lower1.5ex\hbox{`}\hidewidth\crcr\unhbox0}}} \def\cprime{$'$}
  \def\cprime{$'$} \def\cprime{$'$} \def\cprime{$'$}

\end{document}